\documentclass[12pt]{amsart}
\usepackage{amsfonts,amssymb,amscd,amstext,mathrsfs}
\usepackage[utf8]{inputenc}
\usepackage{hyperref}
\usepackage{verbatim}

\usepackage{times}
\usepackage{enumerate}
\usepackage[up,bf]{caption}

\input xy
\xyoption{all}

\newtheorem{theorem}{Theorem}[section]
\newtheorem{proposition}[theorem]{Proposition}

\newtheorem{corollary}[theorem]{Corollary}

\theoremstyle{definition}
\newtheorem{definition}[theorem]{Definition}
\newtheorem{remark}[theorem]{Remark}

\newtheorem{problem}[theorem]{Problem}

\numberwithin{equation}{section}
\numberwithin{figure}{section}

\newcommand\Cscr{\mathscr{C}}

\newcommand\B{\mathbb{B}}
\newcommand\C{\mathbb{C}}
\newcommand\D{\overline{\mathbb D}}
\newcommand\CP{\mathbb{CP}}

\renewcommand\D{\mathbb D}

\newcommand\R{\mathbb{R}}
\newcommand\RP{\mathbb{RP}}

\newcommand\Z{\mathbb{Z}}

\newcommand\igot{\mathfrak{i}}

\renewcommand\igot{\mathfrak{i}}

\renewcommand\imath{\igot}

\newcommand\di{\partial}

\newcommand\dist{\mathrm{dist}}
\renewcommand\span{\mathrm{span}}

\newcommand\tr{\mathrm{tr}}
\newcommand\Hess{\mathrm{Hess}}

\newcommand\CH{\mathrm{CH}}

\def\dist{\mathrm{dist}}
\def\span{\mathrm{span}}

\numberwithin{equation}{section}

\begin{document}
\title{Domains without parabolic minimal submanifolds and \\ weakly hyperbolic domains} 
\author{Franc Forstneri{\v c}}

\address{Franc Forstneri\v c, Faculty of Mathematics and Physics, University of Ljubljana, and Institute of Mathematics, Physics, and Mechanics, Jadranska 19, 1000 Ljubljana, Slovenia}
\email{franc.forstneric@fmf.uni-lj.si}

\thanks{Supported by the European Union (ERC Advanced grant HPDR, 101053085
to F.\ Forstneri\v c) and by grants P1-0291, J1-3005, and N1-0237 from ARRS, 
Republic of Slovenia.}


\date{1 June 2023}
%%%%%


\begin{abstract}
We show that if $\Omega$ is an $m$-convex domain in $\R^n$ for some
$2\le m<n$ whose boundary $b\Omega$ has a tubular neighbourhood of positive radius
and is not $m$-flat near infinity, then $\Omega$ does not contain any immersed 
parabolic minimal submanifolds of dimension $\ge m$.
In particular, if $M$ is a properly embedded nonflat minimal hypersurface in $\R^n$
with a tubular neighbourhood of positive radius then every immersed parabolic 
hypersurface in $\R^n$ intersects $M$. In dimension $n=3$ this holds if $M$ has 
bounded Gaussian curvature function.
We also introduce the class of weakly hyperbolic domains $\Omega$ in $\R^n$,  
characterised by the property that every conformal harmonic map $\C\to\Omega$ 
is constant, and we elucidate their relationship with hyperbolic domains, and domains 
without parabolic minimal surfaces.
\end{abstract}

\maketitle

%
%
%    INTRODUCTION
%
%
\section{Introduction}\label{sec:intro}
This paper is motivated by two closely related lines of developments in the theory 
of minimal surfaces. The first one is the circle of results known as halfspace theorems. 
The second one is the recently introduced hyperbolicity theory for minimal surfaces.  

Concerning the first topic, Xavier \cite{Xavier1984} proved in 1984 that the convex hull 
of a complete nonflat minimal surface in $\R^3$ with bounded Gaussian curvature 
equals $\R^3$. %, so such a surface cannot be contained in a halfspace.
The {\em Strong Halfspace Theorem} of Hoffman and Meeks 
\cite[Theorem 2]{HoffmanMeeks1990IM} says that any two proper minimal surfaces 
in $\R^3$ intersect, unless they are parallel planes.  
By Pacelli Bessa, Jorge, and Oliveira-Filho \cite{BessaJorgeOliveira2001} 
and Rosenberg \cite{Rosenberg2001}, the same conclusion holds for a pair of complete immersed minimal surfaces in $\R^3$ with bounded curvature; 
such surfaces need not be proper in $\R^3$ (see Andrade \cite{Andrade2000PAMS}) 
unless they are embedded (see Meeks and Rosenberg 
\cite[Theorem 2.1]{MeeksRosenberg2008JDG}).
Closer to the topic of this paper, Pacelli Bessa, Jorge, and Pessoa 
proved in \cite{BessaJorgePessoa2021} that an immersed parabolic minimal surface 
in $\R^3$ intersects every complete immersed nonflat minimal surface with bounded curvature. Related developments 
concern %the barrier principle and 
the maximum principle at infinity; see the surveys by Meeks and P\'erez 
\cite{MeeksPerez2004SDG,MeeksPerez2012Survey}   
and the recent paper by Gama et al.\ \cite{GamaLiraMariMedeiros2022}. 

In this paper we find geometric conditions on the boundary of 
a domain $\Omega$ in $\R^n$ for $n\ge 3$ implying that $\Omega$ does not contain any 
immersed (not necessarily proper or complete) parabolic minimal submanifolds of 
a given dimension; see Theorem \ref{th:main}.  In particular, we obtain a halfspace 
theorem for a pair of minimal hypersurfaces in $\R^n$, one of which is parabolic 
and the other one is properly embedded  and has a tubular neighbourhood 
of positive radius; see Corollary \ref{cor:main}.
When $n=3$, this  coincides with a special case of the aforementioned result 
of Pacelli Bessa et al.\ \cite[Theorem 1.1]{BessaJorgePessoa2021}; 
see Corollary \ref{cor:main3}. 

% Insertion in the revision
Parabolicity of an open Riemannian manifold is important 
from the point of view of potential theory.
Determining whether such a manifold is parabolic or hyperbolic 
is a classical question, the so-called type problem. The survey by Grigor'yan \cite{Grigoryan1999BAMS} is a standard reference. The extrinsic case, considering the type problem for a submanifold of a Riemannian manifold, is natural and interesting as well, in particular when the ambient manifold is a Euclidean space. The case of minimal surfaces in 
$\R^n$ has been widely studied. For minimal submanifolds of dimension higher than $2$ there is not much literature, the reason probably being that every complete minimal submanifold of dimension at least $3$ in a Euclidean space is hyperbolic; see 
Markvorsen and Palmer \cite[Theorem 2.1]{MarkvorsenPalmer2003}. 
%So, in this case, it is crucial that in our Theorem \ref{th:main} and Corollary \ref{cor:main} we do not assume the submanifolds to be complete. 
%Theorem A in the paper \cite{MarkvorsenPalmer2003} by also serves as motivation for the study of the type problem in arbitrary dimension.

The absence of parabolic minimal submanifolds in a given domain indicates 
that the domain is small or tight for minimal submanifolds. Another measure of 
smallness is hyperbolicity (as measured by minimal surfaces), 
a notion introduced in the recent paper
by Forstneri\v c and Kalaj \cite{ForstnericKalaj2021}, and developed further 
by Drinovec Drnov\v sek and Forstneri\v c \cite{DrinovecForstnericPAMQ}. 
A domain $\Omega\subset\R^n$ is hyperbolic if for every point $p\in \Omega$ 
there are a neighbourhood $U\subset\Omega$ of $p$ and a constant $c_p<\infty$ 
such that every conformal harmonic disc $f:\D=\{z\in \C:|z|<1\}\to\Omega$ with 
$f(0)\in U$ satisfies $\|df_0\|\le c_p$. This is a close analogue of Kobayashi hyperbolicity 
in complex analysis; see Kobayashi \cite{Kobayashi1967,Kobayashi2005}. 
In Section \ref{sec:weakhyp} we introduce the class of 
{\em weakly hyperbolic domains} --- 
domains $\Omega\subset\R^n$ $(n\ge3)$ without nonconstant conformal harmonic maps 
$\C\to \Omega$. Weak hyperbolicity is an analogue of Brody hyperbolicity,  
which excludes nonconstant holomorphic maps from $\C$ to a given 
complex manifold. A hyperbolic domain is also weakly hyperbolic;
the converse holds on convex domains (see Proposition \ref{prop:convex}) 
but fails in general. In fact, we show that the class of weakly hyperbolic domains 
properly contains the other two mentioned classes,
and there is a nonhyperbolic domain in $\R^3$ without parabolic minimal surfaces
(see Proposition \ref{prop:relationship}). On the other hand, it remains an open problem 
whether there exists a hyperbolic domain containing a parabolic minimal surface.

%
%    SECTION: MAIN RESULT - NO PARABOLIC MINIMAL SURFACES
%
\section{Excluding parabolic minimal submanifolds}\label{sec:main}
Let $m\ge 2$ be an integer. A  connected $m$-dimensional Riemannian manifold $(R,g)$ 
is said to be {\em parabolic} if every negative subharmonic function on $R$ 
is constant. A connected manifold immersed into a Euclidean space $\R^n$
is said to be parabolic if it is parabolic in the metric induced from the Euclidean metric on $\R^n$
by the immersion. On a surface $R$, parabolicity only depends on the conformal class of the metric.
Every compact conformal surface punctured at finitely many points is parabolic;
an example is $\C=\CP^1\setminus\{\infty\}$. See Grigor$'$yan \cite{Grigoryan1999BAMS}
for more information. The following result is proved in Section \ref{sec:proof}.  

%
%   MAIN THEOREM
%
\begin{theorem}\label{th:main}
Let $2\le m<n$ be integers. Assume that $\Omega$ is a domain in $\R^n$ 
with $\Cscr^2$ boundary $M=b\Omega$ whose principal curvatures 
$\nu_1\le \nu_2\le \cdots\le \nu_{n-1}$ satisfy $\nu_1+\nu_2+\cdots+\nu_m\ge 0$ 
at every point, the set of $m$-flat points $\{p\in M:\nu_j(p)=0\ \text{for}\ j=1,\ldots, m\}$ 
has bounded interior in $M$, % (this condition is unnecessary if $3\le m=n-1$), 
and there is an $\epsilon>0$ such that every point in the $\epsilon$-neighbourhood 
of $M$ has a unique nearest point in $M$. Then, $\Omega$ does not contain any 
parabolic immersed minimal submanifolds of dimension $\ge m$. In particular, 
if these conditions hold for $m=2$ then every conformal harmonic map $R\to \Omega$ 
from a parabolic conformal surface $R$ is constant. 
\end{theorem}

The minimal submanifolds in the theorem are not assumed to be 
complete or proper. The hypothesis on the set of $m$-flat points of $b\Omega$
cannot be omitted unless $3\le m=n-1$. 
A counterexample is a halfspace $H$ of $\R^n$, which contains 
parabolic minimal submanifolds of dimensions $m\in\{2,\ldots,n-2\}$, or $m=2$ if $n=3$,
contained in hyperplanes parallel to $bH$. (Note that $\R^m$ with the flat metric
is parabolic for $m=2$, but is hyperbolic if $m\ge 3$, carrying the negative subharmonic 
function $-1/|x|^{n-2}$.) When $m=2$, the minimal surfaces in the theorem 
are allowed to have isolated branch points, 
% (See Sect.\ \ref{sec:proof} for a more precise explanation.) 
and in this case the result is essentially optimal since many open conformal surfaces 
of hyperbolic type admit nonconstant bounded (and even complete)
conformal harmonic maps to $\R^3$.
This holds in particular for any bordered conformal surface; 
see \cite[Chapter 7]{AlarconForstnericLopez2021} for a survey of this topic. 

A domain $\Omega\subset \R^n$ with $\Cscr^2$ boundary %$M=b\Omega$ 
whose principal curvatures $\nu_1\le \nu_2\le \cdots\le \nu_{n-1}$ from the inner side 
satisfy $\nu_1+\nu_2+\cdots+\nu_m \ge 0$ at every point of $b\Omega$ 
is said to be {\em $m$-convex}, and if $m=2$ it is also called {\em minimally convex}; 
see \cite[Definition 8.1.9]{AlarconForstnericLopez2021} or Definition \ref{def:m-flat}. 
By \cite[Theorem 8.1.13]{AlarconForstnericLopez2021} this holds % the domain $\Omega$ is $m$-convex
if and only if there exist a neighbourhood $U\subset \R^n$ of $b\Omega$ and a $\Cscr^2$ function
$\rho: U\to \R$ such that  $\Omega\cap U=\{\rho<0\}$,  $d\rho\ne 0$ on $b\Omega\cap U=\{\rho=0\}$, and
\begin{equation}\label{eq:trace-bD}
	\tr_\Lambda \Hess_\rho(x) \ge 0 \ \ \text{for every point $x\in b\Omega$ and $m$-plane
	$\Lambda \subset T_x b\Omega$}.
\end{equation}
Here, $\tr_\Lambda \Hess_\rho(x)$ denotes the trace of the restriction to $\Lambda$ of the Hessian form of $\rho$ at $x$. 
It was shown in \cite{Forstneric2022BSM} that a bounded $m$-convex domain with $\Cscr^2$ boundary in $\R^n$
admits an $m$-plurisubharmonic defining function,
i.e., one satisfying condition \eqref{eq:trace-bD} for every point $x\in \overline \Omega$ and $m$-plane 
$\Lambda\subset \R^n$ (see \eqref{eq:trace}). The main new point shown in this paper 
is that an unbounded domain $\Omega$ as in Theorem \ref{th:main} admits a defining function 
which is $m$-plurisubharmonic on $\overline \Omega$, 
obtained by convexifying the signed distance function to $b\Omega$; 
see Proposition \ref{prop:main}. This is the key fact of independent interest 
used in the proof of Theorem \ref{th:main}.

Note that $m$-convex domains play a major role in the theory of minimal submanifolds. 
% see \cite[Chapter 8]{AlarconForstnericLopez2021}.
In particular, every bounded minimally convex domain contains many properly immersed 
minimal surfaces parameterized by an arbitrary bordered conformal surface;  
see \cite[Theorems 8.3.1 and 8.3.4]{AlarconForstnericLopez2021} for the orientable case and
\cite[Theorem 6.9]{AlarconForstnericLopezMAMS} for the nonorientable one.
Such domains also form barriers for minimal submanifolds; see 
Jorge and Tomi \cite{JorgeTomi2003} and Gama et al.\ \cite{GamaLiraMariMedeiros2022}, and 
the references therein. In \cite{GamaLiraMariMedeiros2022} the authors proved the 
maximum principle at infinity in a very general context including parabolic minimal $m$-varifolds in $m$-convex domains.    
Nevertheless, I was unable to deduce Theorem \ref{th:main} from \cite[Theorem 1.1]{GamaLiraMariMedeiros2022} 
which holds under weaker hypotheses on $b\Omega$ but seemingly stronger hypotheses on the 
minimal submanifolds. % given by the growth estimates \cite[(1.4) and (1.7)]{GamaLiraMariMedeiros2022}. 
%mainly because of the growth estimates (1.4) and (1.7) on the minimal submanifold $\Sigma\subset \Omega$ in their paper. It seems that the reduction holds if $\Sigma$ is parabolic and proper in $\Omega$. 
% What they prove is the max principle at infinity, (1.8). They assume the growth estimates (1.4) and  (1.7) on Sigma. They make comments on bottom of page 313/top of page 314 by which other conditions can these estimates be replaced:
%A replacement for (1.4) is either (i) stochastic completeness, or (ii) properness.
%A replacement for (1.7) is parabolicity, condition (iii).
% So, if I understand well, the two conditions together apply to properly immersed parabolic submanifolds. My proof does not need properness.
%In some sense, there is a tradeoff between the assumptions on $b\Omega$ and on the minimal submanifolds in $\Omega$.
%Our proof of Theorem \ref{th:main} is considerably simpler that the one in \cite{GamaLiraMariMedeiros2022}.
%, firstly since we are in Euclidean space, and secondly due to the stronger assumption that $b\Omega$ has a tubular neighbourhood of positive radius. 

The case of particular interest is when the boundary $M=b\Omega$ satisfies $\nu_1+\cdots+\nu_{n-1}=0$, so 
it is a minimal hypersurface. If $M$ is not a hyperplane then 
the set of its $(n-1)$-flat points has empty interior, and we obtain the following corollary to Theorem \ref{th:main}.

%
%   MAIN COROLLARY FOR MINIMAL HYPERSURFACES
%
\begin{corollary}\label{cor:main}
Assume that $M$ is a properly embedded minimal hypersurface in $\R^n$ for $n\ge 3$ 
such that for some $\epsilon>0$ every point in the $\epsilon$-neighbourhood 
of $M$ has a unique nearest point in $M$. Then every immersed parabolic minimal 
hypersurface $R\to\R^n$ intersects $M$, unless $n=3$, $M$ is a plane in $\R^3$,
and the image of $R$ is contained in a plane parallel to $M$.
\end{corollary}

The condition in Theorem \ref{th:main} and Corollary \ref{th:main},
that the $\Cscr^2$ hypersurface $M=b\Omega$ admit a tubular neighbourhood 
of positive radius $\epsilon>0$, is nontrivial when $\Omega$ is unbounded, 
which is the only case of interest. If this holds, one says that $M$ has 
{\em positive reach}, a terminology introduced by Federer \cite[Sect.\ 4]{Federer1959}. 
The {\em reach} of $M$ is the supremum of the numbers $\epsilon>0$ having this property.
A $\Cscr^2$ hypersurface with positive reach has bounded principal curvatures
(see the proof of Proposition \ref{prop:main}). The converse holds in several cases 
of interest. Federer proved \cite[Lemma 4.11]{Federer1959} that a graph in $\R^n$ 
over a domain in $\R^{n-1}$ with sufficiently nice boundary, such that the gradient 
of the graphing function is Lipschitz, has positive reach. Every compact piece of an 
embedded $\Cscr^2$ hypersurface has positive reach.
%, so the problem lies in the behaviour of the hypersurface near infinity. 

If $M$ is a complete embedded minimal surface in $\R^3$ 
of finite total Gaussian curvature, then every end of $M$ is a graph over the complement 
of a disc in $\R^2$ whose graphing function has bounded second order partial derivatives, 
and the ends are well separated (see Jorge and Meeks \cite{JorgeMeeks1983T}). 
Hence, every such surface has positive reach. 
It was shown by Meeks and Rosenberg \cite[Theorem 5.3]{MeeksRosenberg2008JDG}
(see also \cite[Corollary 2.6.6]{MeeksPerez2012Survey}) 
that a complete embedded minimal surface with bounded curvature in $\R^3$ is proper and
has positive reach. This gives the following corollary to Theorem \ref{th:main},
originally due to Pacelli Bessa, Jorge, and Pessoa 
\cite[Theorem 1.1]{BessaJorgePessoa2021}.

%
%   MAIN COROLLARY, dim = 3
%
\begin{corollary}\label{cor:main3}
The image of a nonconstant conformal harmonic map $R \to\R^3$ 
(possibly with branch points) from a parabolic conformal surface 
intersects every properly embedded nonflat minimal surface 
$M$ in $\R^3$ with bounded Gaussian curvature. 
\end{corollary}

%
%  Remark on other papers
%
\begin{remark}
In \cite[Theorem 1.1]{BessaJorgePessoa2021}, Corollary \ref{cor:main3}
is proved under the weaker assumption that $M$ is a complete nonflat immersed 
minimal surface in $\R^3$ of bounded curvature.
(Unlike in the case when $M$ is embedded, such a surface need not be proper 
in $\R^3$ as shown by Andrade \cite{Andrade2000PAMS}.) Their proof for the 
immersed case is more involved. I thank G.\ Pacelli Bessa and L.\ F.\ Pessoa 
for having pointed out their work after an earlier version of this preprint was posted. \end{remark}

If one imposes stronger conditions on a minimal surface in a domain $\Omega\subset\R^3$, 
then the conclusion of Corollary \ref{cor:main3} holds under weaker conditions on $b\Omega$. For example, if $\Omega\subsetneq \R^3$ is a smoothly bounded 
minimally convex domain and $R\subset \Omega$ is a complete immersed minimal surface
of finite total Gaussian curvature (hence proper in $\R^3$), then $R$ is a plane and $\Omega$ is a slab or a halfspace 
(see \cite[Theorem 1.16]{AlarconDrinovecForstnericLopez2019TAMS}). 
This also follows from the results in \cite{GamaLiraMariMedeiros2022}.

Corollary \ref{cor:main3} applies to every complete embedded nonflat minimal surface 
$M$ in $\R^3$ of finite total Gaussian curvature, as well as to many minimal surfaces 
of infinite total curvature such as the standard helicoid, the helicoids of positive genera constructed 
by Hoffman, Traizet, and White \cite{HoffmanTraizetWhite2016}, and Riemann's minimal surfaces
\cite{MeeksPerez2016}. It was shown by Meeks, P\'erez, and Ros \cite[Theorem 1]{MeeksPerezRos2004IM} 
that every properly embedded minimal surface in $\R^3$ of finite genus has bounded curvature,
so Corollary \ref{cor:main3} holds for such surfaces. Furthermore, 
every periodic properly embedded minimal surface in $\R^3$ whose fundamental domain 
is of finite topological type satisfies Corollary \ref{cor:main3} by results of 
Meeks and Rosenberg \cite{MeeksRosenberg1996CMH}; examples include Scherk's surfaces.
We refer to Meeks and P\'erez \cite{MeeksPerez2012Survey} for more information on this topic.

%
%   PROBLEM
%
Corollary \ref{cor:main3} gives a partial negative answer to 
\cite[Problem 1.16]{DrinovecForstneric2023JMAA}, 
asking whether the complement of a nonflat properly embedded minimal surface in $\R^3$ 
contains a minimal surface parameterized by $\C$. The case of unbounded curvature 
(and hence necessarily of infinite genus by Meeks, P\'erez and Ros
\cite[Theorem 1]{MeeksPerezRos2004IM}) remains open:

\begin{problem}\label{prob:DDF}
Let $M$ be a properly embedded minimal surface in $\R^3$ with unbounded curvature. 
Is there a nonconstant conformal harmonic map $\C\to\R^3$ whose image avoids $M$?
\end{problem}

\section{Proof of Theorem \ref{th:main}}\label{sec:proof}
We begin with preliminaries. Recall that a map $f=(f_1,\ldots,f_n):U\to \R^n$ from an 
open set $U\subset \C$ is said to be conformal if it preserves angles at every immersion 
point; equivalently, if $z=x+\imath y$ is a complex coordinate on $U$ then the partial 
derivatives of $f$ satisfy $f_x\,\cdotp f_y=0$ and $|f_x|=|f_y|$. 
Here, the dot denotes the Euclidean inner product and $|\,\cdotp|$ the Euclidean length.
These conditions imply that the rank of $f$ at any point is either two (an immersion point) 
or zero (a branch point). 
The analogous definition applies to maps from any conformal surface in local isothermal 
coordinates. The image of a nonconstant conformal harmonic map from an open 
conformal surface is a minimal surface with isolated branch points; conversely, every 
minimal surface with isolated branch points is of this form. 
For background, see e.g. the monographs \cite{AlarconForstnericLopez2021,ColdingMinicozzi1999,Osserman1986} 
and the surveys \cite{AlarconForstneric2019JAMS,MeeksPerez2011,MeeksPerez2012Survey}.

%
%  Change in the revision: replaced $d_M$ by $\dist(x,M)$
%
Given a subset $M$ of $\R^n$, we denote by $\dist(\cdotp,M)$ the Euclidean distance 
function to $M$:
\[	
	\dist(x,M)=\inf\{|x-p|:p\in M\},\quad x\in\R^n.
\]
For $p\in\R^n$ and $\epsilon>0$ we let $\B(p,\epsilon)=\{x\in\R^n:|x-p|<\epsilon\}$ and 
\begin{equation}\label{eq:Vepsilon}
	V_\epsilon(M)=\bigcup_{p\in M} \B(p,\epsilon)= \left\{x\in \R^n: \dist(x,M)<\epsilon\right\}\!.
\end{equation}
% we denote the open $\epsilon$-neighbourhood of $M$. 

Let $M$ be a properly embedded connected hypersurface in $\R^n$. Such $M$ is 
orientable and its complement $\R^n\setminus M=M^+\cup M^-$ consists of a pair of connected domains.
The signed distance function to $M$ is defined by 
\begin{equation}\label{eq:SD}
	\delta_M(x) = \begin{cases} \phantom{-} \dist(x,M), & \text{if}\ x\in M_+ \cup M; \\
					   	    -\dist(x,M), & \text{if}\ x\in M_-.
			\end{cases}
\end{equation}
If $V\subset \R^n$ is an open neighbourhood of $M$ such that every point $x\in V$ 
has a unique nearest point $p=\xi(x)\in M$, and if $M$ is of class $\Cscr^r$ for some 
$r\in \{2,3,\ldots,\infty,\omega\}$, then $\delta_M$ is also of class $\Cscr^r$ on $V$
(see Gilbarg and Trudinger \cite[Lemma 14.16]{GilbargTrudinger1983} or Krantz and Park \cite{KrantzPark1981}). 
%If this holds on the $\epsilon$-neighbourhood $V_\epsilon(M)$ \eqref{eq:Vepsilon} of $M$ for some $\epsilon>0$, we say that $M$ has {\em positive reach}, and the {\em reach} of $M$ is the supremum of such numbers $\epsilon>0$ \cite[Sect.\ 4]{Federer1959}.

We recall the following notion; see 
\cite[Sections 8.1.4--8.1.5]{AlarconForstnericLopez2021}.
% and in particular \cite[Theorem 8.1.13]{AlarconForstnericLopez2021}. 

%
%  m-convex and m-flat
%
\begin{definition}\label{def:m-flat}
An oriented embedded $\Cscr^2$ hypersurface $M\subset \R^n$ is {\em $m$-convex} 
for some $m\in \{1,2,\ldots,n-1\}$ if its principal curvatures 
$\nu_1\le \nu_2\le \cdots \le \nu_{n-1}$ satisfy
\begin{equation}\label{eq:m-convex}
		\nu_1(p)+\nu_2(p)+\cdots +\nu_m(p)\ge 0\ \ \text{for all $p\in M$}.
\end{equation}
The hypersurface $M$ is {\em strongly $m$-convex} at $p\in M$ if strong inequality holds in \eqref{eq:m-convex}. 
A point $p\in M$ is {\em $m$-flat} if $\nu_j(p)=0$ for $j=1,\ldots,m$. We say that 
{\em $M$ is not $m$-flat at infinity} if the set $\{p\in M: \nu_j(p)=0\ \text{for}\ j=1,\ldots,m\}$
has bounded relative interior in $M$.
\end{definition}

Given a $\Cscr^2$ function $\rho:\Omega\to \R$ on a domain $\Omega\subset \R^n$,
we denote by $\Hess_\rho(x)$ the Hessian of $\rho$ at the point $x\in \Omega$, 
i.e., the quadratic form on $T_x\R^n=\R^n$ represented by the matrix 
$\big(\frac{\di^2\rho}{\di x_i\di x_j}(x)\big)$ of second order partial derivatives of 
$\rho$ at $x$. For $1\le m<n$ let $G_m(\R^n)$ denote the Grassman manifold 
of $m$-planes in $\R^n$. Given $\Lambda\in G_m(\R^n)$, we denote by 
$\tr_\Lambda \Hess_\rho(x)\in\R$ the trace of the restriction of the Hessian 
$\Hess_\rho(x)$ to $\Lambda$. A $\Cscr^2$ function $\rho:\Omega\to\R$ is 
said to be {\em $m$-plurisubharmonic} if
\begin{equation}\label{eq:trace}
	\tr_\Lambda \Hess_\rho(x) \ge 0\ \ \text{for all $x\in \Omega$ and $\Lambda\in G_m(\R^n)$}.
\end{equation}
If $\lambda_1(x) \le\lambda_2(x)\le\cdots \le \lambda_n(x)$ are the eigenvalues of  
$\Hess_\rho(x)$, then \eqref{eq:trace} is equivalent to
\[
	\lambda_1(x)+\cdots+\lambda_m(x) \ge 0 \ \ \text{for every}\ x\in \Omega.
\]
In fact, 
$
	\lambda_1(x)+\cdots+\lambda_m(x) = \inf_\Lambda \tr_\Lambda \Hess_\rho(x).
$
(See \cite[Sect.\ 8.1]{AlarconForstnericLopez2021} or \cite{HarveyLawson2013IUMJ}.) 
The main point is that a function $\rho:\Omega\to\R$ is 
$m$-plurisubharmonic if and only if its restriction 
to every $m$-dimensional minimal submanifold in $\Omega$ is a subharmonic function 
on the submanifold (see \cite[Proposition 8.1.2]{AlarconForstnericLopez2021}). 
% A $2$-plurisubharmonic function $\rho$ is also called {\em minimal plurisubharmonic}.
% this holds if and only if for every conformal harmonic map $f:R\to\Omega$ from a conformal surface $R$ the composition $\rho \circ f$ is a subharmonic function on $R$.  
%
%  Added in the revision
%
The notion of an $m$-plurisubarmonic function extends to upper semicontinuous functions 
by asking that for any $m$-dimensional affine subspace $L$ of $\R^n$, 
the restriction of the function to $L\cap \Omega$ is a subharmonic function
(see \cite[Definition 8.1.1]{AlarconForstnericLopez2021}). Such a function can be 
approximated from above on any relatively compact subdomain of $\Omega$
by smooth $m$-plurisubarmonic functions 
(see \cite[Proposition 8.1.6]{AlarconForstnericLopez2021}). In this paper, we shall
use continuous $m$-plurisubharmonic functions obtained by taking the maximum
of a smooth $m$-plurisubharmonic function and a constant, and in this case, 
smoothing can be obtained globally in a simple way; see the last part of the
proof of Proposition \ref{prop:main}.

%
%  IDEA OF PROOF (PERHAPS USE IN THE INTRODUCTION)
%
\begin{comment}
The idea of proof of Theorem \ref{th:main} is the following. Assuming that $M=b\Omega$ has a tubular
$\epsilon$-neighbourhood $V_\epsilon(M)$ as in the theorem, the signed distance function 
$\delta_M$ to $M$ \eqref{eq:SD} is of class $\Cscr^2$ on $V_\epsilon(M)$ 
(see Gilbarg and Trudinger \cite[Lemma 14.16]{GilbargTrudinger1983} or Krantz and Park \cite{KrantzPark1981}). 
If the smallest principal curvature $\nu_1$ of $M$ is bounded from below, 
then we can convexify $\delta_M$ to find a minimal plurisubharmonic function $\rho:\Omega\to [-c,0)$ for some $c>0$
whose level sets $\{\rho=t\}$ for $t\in (-c,0)$ coincide with the level sets of $\delta_M$.
If $f:R\to\Omega$ is a nonconstant conformal harmonic map from a parabolic conformal surface, 
then $\rho\circ f$ is a bounded subharmonic function on $R$, hence constant. It follows that the minimal
surface $f(R)$ lies in the domain $\{x\in\Omega:\rho(x)\le -c\}$ or in a level set $\{\rho=t\}$ for some $t\in (-c,0)$. 
Since the set $\{\nu_1=\nu_2=0\}$ has bounded relative interior in $M$, the second case is impossible. 
This shows that $f(R)$ has positive distance to $M$ which only depends in the 
geometry of $M$, and by translating $M$ we infer that such $f$ cannot exist. 
\end{comment}

Theorem \ref{th:main} is a consequence of the following proposition of independent 
interest. 

%
%   MAIN PROPOSITION
%
\begin{proposition}\label{prop:main}
Let $\Omega \subset \R^n$ be a domain with $\Cscr^r$ boundary for
some $r\ge 2$ satisfying the assumptions in Theorem \ref{th:main}. 
There is a function $\rho: \overline \Omega\to (-\infty,0]$ 
of class $\Cscr^r(\overline \Omega)$ which is $m$-plurisubharmonic
on $\Omega$, it vanishes on $M=b\Omega$, 
it satisfies $d\rho\ne 0$ on $\rho^{-1}((-1,0])$, and for every $t\in (-1,0)$  
the level set $M_t=\{\rho=t\}$ coincides with a level set of the signed distance 
function $\delta_M$ \eqref{eq:SD}, $M_t$ is $m$-convex, 
it is strongly $m$-convex at every $m$-nonflat point, and $M_t$ 
is not $m$-flat near infinity.
\end{proposition}

We point out that, for the proof of Theorem \ref{th:main}, 
it suffices to find a continuous $m$-plurisubharmonic function 
$\rho:\overline \Omega\to (-\infty,0]$ which is of class $\Cscr^2$ on an
interior collar around $M=b\Omega$ and has the other stated properties.
Such a function is given by \eqref{eq:rho0}. 

%
%   PROOF OF THE PROPOSITION 
%
\begin{proof} 
By Gilbarg and Trudinger \cite[Lemma 14.16]{GilbargTrudinger1983} 
or Krantz and Parks \cite{KrantzPark1981}, the conditions on 
$M=b\Omega$ imply that the signed distance function 
$\delta=\delta_M$ to $M$ \eqref{eq:SD} is 
of class $\Cscr^r$ on the $\epsilon$-neighbourhood $V_\epsilon(M)$ of $M$
(see \eqref{eq:Vepsilon}). We choose the sign so that $\delta<0$ on 
\begin{equation}\label{eq:Cepsilon}
	C_\epsilon:=\Omega\cap V_\epsilon(M)= \{x\in \Omega: \dist(x,M)<\epsilon\}.
\end{equation}
We recall some further properties of $\delta$, referring to 
Bellettini \cite[Theorem 1.18, p.\ 14]{Bellettini2013} 
and Gilbarg and Trudinger \cite[Section 14.6]{GilbargTrudinger1983}.
There is a projection $\xi:V_\epsilon(M)\to M$ of class $\Cscr^{r-1}$ such that for every 
$x\in V_\epsilon(M)$ the point $p=\xi(x)\in M$ is the unique nearest point to $x$ in $M$. 
The gradient $\nabla \delta$ has constant norm $|\nabla \delta|=1$ on $V_\epsilon(M)$, 
and it has constant value on the intersection of $V_\epsilon(M)$ with the normal line 
$N_p=p+\R\,\cdotp \nabla \delta(p)$ at $p\in b\Omega=M$. There an orthonormal 
basis $(v_1,\ldots,v_{n-1},v_n=\nabla \delta(p))$ of $\R^n$ 
such that the vectors $v_1,\cdots,v_{n-1}$ span $T_pM$, 
and in this basis the (symmetric) matrix $A(p)$ of $\Hess_\delta(p)$ is diagonal:
\begin{equation}\label{eq:Ap}
	A(p) v_j=\nu_j v_j\ \ \text{for $j=1,\ldots, n-1$}; \quad A(p)v_n=0, 
\end{equation}
where the numbers $\nu_1\le \nu_2\le \cdots \le \nu_{n-1}$ 
are the principal normal curvatures of $M$ at $p$ from the interior side. 
For any point $x\in N_p \cap V_\epsilon(M)$ the same basis $(v_1,\ldots,v_n)$ 
diagonalizes $\Hess_\delta(x)$, with the corresponding eigenvalues 
\begin{equation}\label{eq:nux}
	\nu_j(x)=\frac{\nu_j}{1 + \delta(x) \nu_j}\ \ \text{for $j=1,\ldots, n-1$}; 
	\quad A(p)v_n=0.
\end{equation}
If $\nu_j>0$ for some $j$ then the above shows that $1+\delta(x)\nu_j>0$ for all $x\in C_\epsilon\cap N_p$.
Since $\delta(x)$ assumes all values  in $(-\epsilon,0)$ on the set $C_\epsilon$ \eqref{eq:Cepsilon}, this 
implies $1-\epsilon\nu_j\ge 0$ and hence $\nu_j\le 1/\epsilon$ for all such $j$. From this
and the hypothesis $\nu_1+\nu_2+\cdots+\nu_m\ge 0$  %(i.e., $M$ is $m$-convex)
it also follows that $\nu_j\ge -(m-1)/\epsilon$ for every $j=1,\ldots,n-1$. 
Summarising, we have that
\begin{equation}\label{eq:bounds}
	-(m-1)/\epsilon \le \nu_j(p) \le 1/\epsilon\ \ \text{for all $j=1,\ldots,n-1$ and $p\in M$}.
\end{equation}
That is, the hypersurface $M=b\Omega$ has bounded principal curvatures. 
The worst case estimate in the first inequality in \eqref{eq:bounds} occurs only when 
$\nu_1=-(m-1)/\epsilon$ and $\nu_j=1/\epsilon$ for $j=2,\ldots,m$, so we also conclude that
\begin{equation}\label{eq:lowerbound}
	\sum_{\nu_j(p)\le 0}  \nu_j(p) \ge -\frac{m-1}{\epsilon}\ \ \text{for all $p\in M$.}
\end{equation}
Consider the family of domains 
\[
	\Omega_t=\{x\in C_\epsilon:\delta(x)<t\}\cup (\Omega\setminus C_\epsilon)\quad
	\text{for $t\in (-\epsilon,0]$.}
\] 
As $t$ increases towards $0$, the domains $\Omega_t$ increase to $\Omega_0=\Omega$. 
The tangent space to the hypersurface $b\Omega_t=\{\delta=t\}$ 
at the point $x=p+t \nabla\delta(p)$ is spanned by the same vectors $v_1,\ldots,v_{n-1}$ 
as $T_p M$ (see \eqref{eq:Ap}), and the numbers $\nu_j(x)$ in \eqref{eq:nux} for 
$j=1,\ldots,n-1$ are the principal curvatures of $b\Omega_t$ at $x$. As $\delta(x)$ 
decreases (we move away from $M$), each of the curvatures $\nu_j(x)$ strictly increases 
unless $\nu_j(p)=0$ (in whihc case $\nu_j(x)=0$), and it does not change the sign. 
This implies that for every $x=p+t \nabla\delta(p) \in b\Omega_t$ with 
$t\in (-\epsilon,0)$ we have 
\begin{eqnarray}
	\label{eq:ordered}
	& \nu_1(x)\le \nu_2(x)\le \cdots \le \nu_{n-1}(x), \ \ \nu_j(x) \ge \nu_j(p)\ \ 
	\text{for $j=1,\ldots,n-1$},\ \ \text{and} \\
        \label{eq:Hge0}
	& H(x):=\nu_1(x)+\nu_2(x)+\cdots+\nu_m(x)\ge \nu_1(p)+\nu_2(p)+\cdots+\nu_m(p)\ge 0,
\end{eqnarray}
where the first inequality in \eqref{eq:Hge0} is equality if and only if the point 
$p\in M$ is $m$-flat (see Definition \ref{def:m-flat}). 
Indeed, the function $H$ increases as we move away from 
$M=b\Omega$ into $\Omega$, and it vanishes
on $N_p\cap C_\epsilon$ if and only if the point $p\in M$ is $m$-flat. 
We conclude that for $t\in (-\epsilon,0)$ the hypersurface $M_t=b\Omega_t$ is $m$-convex, 
and it is strongly $m$-convex at $x=p+t\nabla \delta(p)\in C_\epsilon$ if and only if the 
point $p=\xi(x)\in M$ is not $m$-flat. In particular, since $M$ is not $m$-flat near infinity
by theassumption, $M_t$ is not $m$-flat near infinity for any $t\in (-\epsilon, 0)$.

From \eqref{eq:nux}--\eqref{eq:ordered} it follows that
for every $x\in C_{\epsilon/2}$ (see \eqref{eq:Cepsilon}) we have that
\begin{eqnarray}\label{eq:nuxest}
	& \displaystyle{-\frac{m-1}{\epsilon} \le \nu_j(x) \le \frac{2}{\epsilon} \ \ \text{for $j=1,\ldots,n-1$, and}} \\
	\label{eq:lowerboundatx}
	& \displaystyle{ \sum_{\nu_j(x) \le 0}  \nu_j(x) \ge -\frac{m-1}{\epsilon}.} 
\end{eqnarray}

Despite the estimate \eqref{eq:Hge0}, the function $\delta$ need not be 
$m$-plurisubharmonic on $C_\epsilon$ due the zero eigenvalue of $\Hess\, \delta$ 
in the normal direction determined by $\nabla \delta$. 
However, we now show that a suitable convexification of $\delta$ on a collar 
$C_{\epsilon_0} \subset\Omega$ in \eqref{eq:Cepsilon} for some 
$0<\epsilon_0<\epsilon/2$ extends to an $m$-plurisubharmonic function 
$\rho:\Omega\to (-\infty,0)$ such that every level set 
\[
	M_t=\{\rho=t\}\ \ \text{for $t\in (-\epsilon_0,0)$}
\]
coincides with a level set $\{\delta=\tau\}$ for some $\tau=\tau(t) \in (-\epsilon, 0)$.
The argument is similar to  \cite[proof of Theorem 1.1]{Forstneric2022BSM},
but we are now dealing with an unbounded domain $\Omega$,
and we obtain uniform estimates of the quantities involved. 

Choose a number $\alpha > (m-1)/\epsilon$. 
Let $h:\R\to\R$ be a smooth, convex, increasing function with 
$h(0)=0$, $\dot h(0)=1$, and $\ddot h(0)=\alpha$. Then, 
\begin{equation}\label{eq:doth}
	h(t)<0\ \ \text{and}\ \ 0 \le \dot h(t)<1\ \ \text{for}\ \ -\infty<t<0. 
\end{equation}
Choose numbers $\epsilon_0,\epsilon'_0$ with 
$0< \epsilon_0<\epsilon'_0<\epsilon/2$ such that 
\begin{equation}\label{eq:ddoth}
	\ddot h(t) > \frac{m-1}{\epsilon} \ \ \text{for $-\epsilon'_0 \le t\le 0$}.
\end{equation}
Consider the function $h\circ \delta: C_\epsilon \to (-\infty,0)$. We have that  
\begin{equation}\label{eq:Hesscomp}
	\Hess_{h\circ \delta} = 
	(\dot h\circ \delta) \Hess_{\delta} 
	+ (\ddot h\circ \delta) \nabla \delta\, \cdotp (\nabla \delta)^T.
\end{equation}
(Here, $(\nabla \delta)^T$ is the transpose of $\nabla \delta$, and 
$\nabla \delta\, \cdotp (\nabla \delta)^T$ is an $n\times n$ matrix.) 
Recall that for $p\in M$ the orthonormal vectors $v_1,\cdots v_{n-1},v_n=\nabla\delta(p)$ 
diagonalize $\Hess_\delta(p)$, where $T_p M=\span\{v_1,v_2,\ldots,v_{n-1}\}$. 
Note that $T_pM$ lies in the kernel of the matrix $\nabla \delta\, \cdotp (\nabla \delta)^T$ 
while $v_n$ is an eigenvector with eigenvalue $1$.  The same basis then diagonalizes 
$\Hess_{h\circ \delta}(x)$ at every point $x=p+ \delta(x) v_n \in N_p\cap C_\epsilon$, 
the eigenvalues corresponding to $v_1,\ldots,v_{n-1}$ get multiplied by 
$\dot h(\delta(x))\ge 0$, and the eigenvalue in the normal direction $v_n$ equals 
$\ddot h(\delta(x))$. Summarising, the eigenvalues of $\Hess_{h\circ \delta}(x)$ 
at any point $x\in C_\epsilon$ equal
\begin{equation}\label{eq:eigenvalues-x}
	\dot h(\delta(x)) \nu_1(x),\ldots, \dot h(\delta(x)) \nu_{n-1}(x),\quad \ddot h(\delta(x)).
\end{equation}
If $x\in C_{\epsilon'_0}$ then $-\epsilon'_0<\delta(x)<0$, and hence by 
\eqref{eq:lowerboundatx}, \eqref{eq:doth} and \eqref{eq:ddoth} we have that
\[
	  \sum_{\nu_j(x)\le 0} \dot h(\delta(x)) \nu_j(x) + \ddot h(\delta(x)) >0.
\]
Together with \eqref{eq:ordered} and \eqref{eq:Hge0} it follows that 
the sum of any $m$ numbers in the list \eqref{eq:eigenvalues-x} is nonnegative 
at every point $x\in C_{\epsilon'_0}$, i.e., the function $h\circ \delta$ is 
$m$-plurisubharmonic on $C_{\epsilon'_0}$. Hence, the continuous function 
$\rho_0:\Omega\to [h(-\epsilon_0),0)$ given by
\begin{equation}\label{eq:rho0}
	\rho_0(x)=\begin{cases} h(\delta(x)), & \text{if}\ x\in C_{\epsilon_0}; \\
					   h(-\epsilon_0), & 
					   \text{if}\ x\in \Omega\setminus C_{\epsilon_0}
		    \end{cases}
\end{equation}
is $m$-plurisubharmonic. Indeed, near $bC_{\epsilon_0}$ we have that 
$\rho_0=\max\{h\circ\delta, h(-\epsilon_0)\}$, and the maximum of two 
$m$-plurisubharmonic functions is $m$-plurisubharmonic 
(see \cite[Sect.\ 6]{HarveyLawson2011ALM}). 

To get a smooth $m$-plurisubharmonic function $\rho \in \Cscr^r(\Omega)$
which agrees with $\rho_0$ on a smaller collar $C_{\epsilon_1}\subset C_{\epsilon_0}$, 
we choose numbers $0<\epsilon_1<\epsilon_2<\epsilon_0$, 
pick a smooth increasing convex function $\chi:\R\to \R$ 
such that $\chi(t)$ is a negative constant for $t\le h(-\epsilon_2)$
and $\chi(t)=t$ for $t\ge h(-\epsilon_1)$, and set $\rho=\chi\circ \rho_0$. 
Clearly, $\rho$ is well-defined on $\overline \Omega$ and of class $\Cscr^r$, 
and it is constant on $\Omega\setminus C_{\epsilon_2}$.
Using the formula \eqref{eq:Hesscomp} for points in $C_{\epsilon_0}$ gives 
\[
	\Hess_{\chi \circ \rho_0} = 
	(\dot \chi \circ \rho_0) \Hess_{\rho_0} 
	+ (\ddot \chi \circ \rho) \nabla \rho_0\, \cdotp (\nabla \rho_0)^T.
\]
Since $\dot \chi \circ \rho_0\ge 0$ and $\ddot \chi \circ \rho\ge 0$, 
$\rho_0$ is $m$-plurisubharmonic on $C_{\epsilon_0}$, 
and the matrix $\nabla \rho_0\, \cdotp (\nabla \rho_0)^T$ is nonnegative definite,
we infer that $\rho=\chi \circ \rho_0$ is of class $\Cscr^r$ and $m$-plurisubharmonic 
on $\overline\Omega$. Clearly, $\rho$ satisfies the conclusion of the proposition 
for all $t\in (h(-\epsilon_1),0)$. For such $t$, we have $\chi(t)=t$ and hence the hypersurface 
$M_t=\{\rho=t\}=\{h\circ \delta=t\}$ equals 
\begin{equation}\label{eq:Mt}	
	M_t=\{\delta=\tau\}\ \ \text{for $\tau=h^{-1}(t)\in (-\epsilon_1,0)$}. 
\end{equation}
Replacing $\rho$ by $c\rho$ with $c=-1/h(-\epsilon_1)>0$, this holds for $t\in (-1,0)$.
\end{proof}

%
%  Change in revision: 
%
\begin{comment}
\begin{remark}\label{rem:smooth}
By smoothing, one can obtain an $m$-plurisubharmonic function of class $\Cscr^2$ 
on $\Omega$ having the properties in Proposition \ref{prop:main}
in a smaller collar $C_{\epsilon_1}\subset C_\epsilon$. 
%This can be done by a convolution with radially symmetric approximate identity with small support. 
However, smoothness of $\rho$ away from $b\Omega$ is inessential in the proof 
of Theorem \ref{th:main}. If $b\Omega$ is of class $\Cscr^r$ for some $r\ge 2$ 
then the same proof gives a function $\rho$ of class $\Cscr^r$ with the stated properties.
\end{remark}
\end{comment}

%
%   PROOF OF THE MAIN THEOREM
%
\begin{proof}[Proof of Theorem \ref{th:main}]
Since an $m$-convex domain is also $k$-convex for any $m<k\le n-1$, 
and every $k$-flat point is also $m$-flat, it suffices to show that the domain 
$\Omega$ in the theorem does not contain any parabolic minimal surfaces 
of dimension $m$.

Let $\rho:\Omega\to (-\infty,0)$ be an $m$-plurisubharmonic function given 
by Proposition \ref{prop:main}. If $f:R\to \Omega$ is an immersed 
minimal submanifold of dimension $m$, then $\rho\circ f$ is a negative function 
on $R$ which is subharmonic in the Riemannian metric induced by $f$ 
from the Euclidean metric on $\R^n$. If in addition this minimal surface is
connected and parabolic, then this function is constant. 
Thus, either $f(R) \subset M_t=\{\rho=t\}$ for some $t\in (-1,0)$, 
or $f(R) \subset \Omega\setminus C_{\epsilon_1}$ where $\epsilon_1>0$ is as in \eqref{eq:Mt} and $C_{\epsilon_1}$ is given by \eqref{eq:Cepsilon}.

Assume that the first possibility holds. The minimal $m$-dimensional submanifold 
$f(R)\subset M_t$ is clearly contained in the set of points where the hypersurface $M_t$ 
is not strongly $m$-convex. By Proposition \ref{prop:main},  
$M_t$ is $m$-flat at every point where it fails to be strongly 
$m$-convex, and the set of its $m$-flat points has bounded interior. 
Since $f(R)$ is parabolic, it cannot be contained in a bounded set, 
so this case is impossible. 

Therefore, $f(R) \subset  \Omega\setminus C_{\epsilon_1}$,
so the distance between $f(R)$ and $b\Omega$ is at least $\epsilon_1>0$. 
If we translate $\Omega$ for the distance $\epsilon_1/2$ in any direction,
the same argument (with the same constants) applies to the translated domain 
$\Omega'$ and its boundary $b\Omega'$. Since $f(R) \subset \Omega'$, 
we infer that $\dist(f(R),b\Omega')\ge \epsilon_1$. 
Since this holds for every translate $\Omega'$ of $\Omega$ for $\epsilon_1/2$, 
we conclude  that $\dist (f(R),b\Omega)\ge 3\epsilon_1/2$. 
Repeating this argument shows that the distance from $f(R)$ to $b\Omega$ 
must be infinite, a contradiction.

If $m=2$, the same argument shows that there are no nonconstant conformal 
harmonic maps $f:R\to\Omega$ (possibly with branch points) from any parabolic 
conformal surface $R$. (In this case, parabolicity does not depend on $f$ 
but only on the conformal structure of $R$.)
\end{proof}

%
%	REMARK ON MY PAPER IN BULL. SCI. MATH. 2022
%
\begin{remark}[Concerning the paper \cite{Forstneric2022BSM}] \label{rem:signeddistance}
By \cite[Theorem 1.1]{Forstneric2022BSM}, every bounded domain $\Omega$ in $\R^n$ 
$(n\ge3)$, whose boundary is of class $\Cscr^{r,\alpha}$ for some $r\ge 2$ and 
$0<\alpha\le 1$ and is $m$-convex for some $m\in \{1,\ldots,n-1\}$, admits an 
$m$-plurisubharmonic defining function of class $\Cscr^{r,\alpha}$. 
The proof in the cited paper refers to the result of Li and Nirenberg 
\cite{LiNirenberg2005} saying that the signed distance function to a hypersurface 
of class $\Cscr^{r,\alpha}$ for such $(r,\alpha)$ is of the same class $\Cscr^{r,\alpha}$ 
near the hypersurface. By using the result of Gilbarg and Trudinger 
\cite[Lemma 14.16]{GilbargTrudinger1983} 
and Krantz and Parks \cite{KrantzPark1981}, we see that 
\cite[Theorem 1.1]{Forstneric2022BSM} also holds in smoothness classes 
$\Cscr^r$ for $r=2,3,\ldots,\infty,\omega$.

The second remark is that the last statement in \cite[Theorem 1.1]{Forstneric2022BSM},
concerning strongly $m$-convex domains,  
is one of several special cases of a result by Harvey and Lawson 
\cite[Theorem 5.12]{HarveyLawson20093CPAM}.
I wish to thank the authors for the relevant communication.
%In applying Theorem 5.12 note that In your case the subequation (or ?Dirichlet set?) F is a cone, so the ray set ?? F equals F .
\end{remark}

%
%   MINIMAL HYPERBOLICITY
%
\section{Weakly hyperbolic domains}\label{sec:weakhyp}

As mentioned in the introduction, the second motivation for this paper comes from the 
recently introduced hyperbolicity theory for domains in $\R^n$ $(n\ge 3)$ in terms of
minimal surfaces. We recall the background.

In 2021, Forstneri\v c and Kalaj  \cite{ForstnericKalaj2021} 
defined on any domain $\Omega$ in $\R^n$ for $n\ge 3$ a Finsler pseudometric 
$g_\Omega:T\Omega= \Omega\times\R^n\to \R_+=[0,+\infty)$, 
called the {\em minimal pseudometric}, and the associated pseudodistance 
$\dist_\Omega:\Omega\times\Omega\to \R_+$ obtained by integrating $g_\Omega$, 
by using conformal harmonic discs $\D=\{z\in \C:|z|<1\}\to \Omega$, in the same way as the 
Kobayashi--Royden pseudometric and pseudodistance are defined on any
complex manifold by using holomorphic discs; see 
\cite{Kobayashi1967,Kobayashi2005,Royden1971}. We recall the definition. 
Let $\CH(\D,\Omega)$ denote the space of conformal harmonic discs $f:\D\to\Omega$,
possibly with branch points. % (i.e., the map is conformal at all immersion points).  
Denote by $z=x+\imath y$ the complex coordinate on $\D$ and by $f_x$ 
the partial derivative of $f$ on $x$. The minimal pseudometric at $p \in\Omega$ 
on a vector $v\in \R^n$ is defined by 
\[
	g_\Omega(p,v) = 
	\inf\bigl\{1/r>0: \exists f\in\CH(\D,\Omega),\ f(0)=p,\ f_x(0)=r v\bigr\} \ge 0.
\]
It follows from definitions that $\dist_\Omega$ is the biggest pseudometric on $\Omega$ 
such that every conformal harmonic map $R\to \Omega$ from a conformal surface $R$ is 
distance decreasing with respect to the Poincar\'e pseudometric on $R$.
Thus, $\dist_\Omega$ describes the fastest possible rate of growth of hyperbolic minimal 
surfaces in $\Omega$. A domain $\Omega$ is said to be {\em hyperbolic} if 
$\dist_\Omega$ is a distance function, and {\em complete hyperbolic} 
if $(\Omega,\dist_\Omega)$ is a complete metric space. 
Every bounded domain is hyperbolic but not necessarily complete hyperbolic. 
On the unit ball of $\R^n$, the minimal metric coincides with the 
Beltrami--Cayley--Klein metric, one of the classical models of hyperbolic geometry; 
see \cite{ForstnericKalaj2021}. We refer to \cite{DrinovecForstnericPAMQ} 
for the basic hyperbolicity theory for domains in Euclidean spaces. 

We now introduce the following weaker notion of hyperbolicity, which is of interest
on unbounded domains where hyperbolicity remains poorly understood. 

%
%  WEAKLY HYPERBOLIC DOMAINS
%
\begin{definition}\label{def:whyp}
A domain $\Omega$ in $\R^n$, $n\ge 3$, is {\em weakly hyperbolic} (for minimal surfaces) 
if every conformal harmonic map $\C\to\Omega$ is constant.
\end{definition}

Weak hyperbolicity is an analogue of Brody hyperbolicity in complex analysis: 
a complex manifold $X$ is said to be {\em Brody hyperbolic} if every holomorphic map
$\C\to X$ is constant. Every Kobayashi hyperbolic manifold is clearly also Brody hyperbolic;
the converse holds on compact complex manifolds (see Brody \cite{Brody1978}) but it fails
in general (see \cite[p.\ 219]{Brody1978}). 

Likewise, every hyperbolic domain in $\R^n$ is weakly hyperbolic 
but the converse fails in general. For the first claim, 
assume that $\Omega\subset \R^n$ is not weakly hyperbolic. 
Pick a nonconstant conformal harmonic map $f:\C\to\Omega$ and a point $a\in \C$ 
where $df_a\ne 0$. The conformal harmonic discs $f_r: \D\to \Omega$ 
given by $f_r(z) = f(a+rz)$ for $r>0$ satisfy  $f_r(0)=f(a)$ and $d(f_r)_0=r df_a$. 
Letting $r\to +\infty$ we get $g_\Omega(f(a),v)=0$ for all $v\in df_a(\C)$, 
so $\Omega$ is not hyperbolic. 
%
% EXAMPLE: WEAKLY HYPERBOLIC DOMAIN WHICH IS NOT HYPERBOLIC
%
The failure of the converse implication is given by part (b) of the following proposition.
Part (c) shows that a weakly hyperbolic domain may contain parabolic minimal surfaces 
whose universal conformal covering surface is the disc.

\begin{proposition}\label{prop:relationship}
\begin{enumerate}[\rm (a)]
\item There is a nonhyperbolic domain in $\R^3$ which does not contain any parabolic minimal surfaces. 
\item In particular, there is a weakly hyperbolic domain in $\R^3$ which is not hyperbolic.
\item For every parabolic domain $D\subsetneq\R^2=\C$ which omits at least two points of $\C$ 
there is a weakly hyperbolic but non\-hyperbolic domain $\Omega_D\subset\R^3$ containing 
$D\times\{0\}$ as a proper minimal surface.
\end{enumerate}
\end{proposition}

\begin{proof}
Given a domain $D\subset\R^2$, we define the domain $\Omega_D\subset\R^3$ by 
\[
	\Omega_D = \bigl\{(x,y,z)\in\R^3: |z|<1,\ z^2(x^2+y^2)<1,\ (x,y)\in D\ \text{if}\ z=0\bigr\}.
\]
We claim that the following assertions hold.
\begin{enumerate}[\rm (i)]
\item The minimal distance in $\Omega_D$ between any pair of points $p,q\in D\times \{0\}$ vanishes.
% so $\Omega_D$ is not hyperbolic. 
\item The image of every nonconstant conformal harmonic map $R\to \Omega_D$ from a 
parabolic conformal surface $R$ is contained in $D\times \{0\}$. 
\end{enumerate}
To prove the first claim (i), assume that $p=(a,b,0)$ and $q=(c,d,0)$.
Set $p_k=(a,b,1/k)$ and $q_k=(c,d,1/k)$ for $k=2,3,\ldots$. 
It is obvious that $\lim_{k\to\infty} \dist_{\Omega_D}(p,p_k) =0$ and $\lim_{k\to\infty} \dist_{\Omega_D}(q,q_k) =0$.
The sequence of linear discs $\{(x,y,1/k):x^2+y^2<k^2\}\subset \Omega_D$  
shows that $\lim_{k\to\infty} \dist_{\Omega_D}(p_k,q_k) =0$. Hence, $\dist_{\Omega_D}(p,q)=0$.

To prove the second claim (ii), assume that $R$ is a parabolic conformal surface
and $f=(f_1,f_2,f_3):R\to\Omega_D$ is a nonconstant conformal harmonic map.
We have that $|f_3|<1$, so $f_3=c$ is constant. If $c\ne 0$ then 
$f_1^2+f_2^2<1/c^2<+\infty$, so $f_1$ and $f_2$ are constant as well, 
a contradiction. Thus,  $c=0$ and hence $f(R)\subset D\times \{0\}$. 

Taking $D$ to be a bounded domain in $\R^2$, property (ii) implies that $\Omega_D$
does not contain any parabolic minimal surfaces. Since $\Omega_D$ is nonhyperbolic 
by property (i), such a domain satisfies condition (a) in the proposition. 
Part (b) is an immediate consequence.

Let $D\subset\C$ be a parabolic domain which omits at least two points of $\C$.
By (ii) the image of every nonconstant conformal
harmonic map $\C\to \Omega_D$ is contained in $D\times \{0\}$, which
contradicts Picard's theorem. Thus, $\Omega_D$ is weakly hyperbolic but it contains
the parabolic minimal surface $D\times \{0\}$. This proves part (c) of the proposition.
\end{proof}

\begin{problem}
Is there a hyperbolic domain $\Omega\subset\R^3$ containing a parabolic minimal surface?
In particular, if $D\subset\R^2=\C$ is a parabolic domain which omits at least two points of $\C$, is 
the domain $D\times (-\epsilon,\epsilon)\subset \R^3$ hyperbolic for some (or all) $\epsilon>0$?
\end{problem}

Note that the domain $(\C\setminus\{0,1\})\times\D$ in $\C^2$
is Kobayashi hyperbolic. The difference between the two cases is that the 
projection of a complex curve to any factor in a product of complex manifolds is holomorphic, 
while the projection of a conformal harmonic map is harmonic but not conformal in general. 
It is easily seen that $\C\setminus\{0,1\}$ contains many harmonic images of $\C$ 
which are conformal at some point, but not everywhere.

%
%   CONVEX DOMAINS
%
The situation is much simpler for convex domains as shown by the following proposition.

\begin{proposition}\label{prop:convex}
For a convex domain $\Omega$ in $\R^n$, $n\ge 3$, the following are equivalent.
\begin{enumerate}[\it (a)]
\item $\Omega$ is complete hyperbolic.
\item $\Omega$ is hyperbolic.
\item $\Omega$ is weakly hyperbolic.
\item $\Omega$ does not contain any affine $2$-plane.
\item $\Omega$ does not contain any parabolic minimal surface.
\end{enumerate}
\end{proposition}

\begin{proof}
The implications $(a) \Rightarrow (b) \Rightarrow (c) \Rightarrow (d)$ and 
$(e)\Rightarrow (d)$ are trivial. The implication $(d)\Rightarrow (a)$ was proved in 
\cite[Theorem 5.1]{DrinovecForstnericPAMQ}, which establishes the equivalence of 
properties $(a)$--$(d)$. It remains to prove $(d)\Rightarrow (e)$. 
It was shown in \cite[proof of Theorem 5.1]{DrinovecForstnericPAMQ} 
that a convex domain $\Omega$ in $\R^n$ satisfying condition $(d)$ is contained 
in the intersection of $n-1$ halfspaces $H_j=\{\ell_j<c_j\}$ $(j=1,\ldots,n-1)$, where 
$\ell_1,\ldots,\ell_{n-1}:\R^n\to\R$ are linearly independent linear functionals and 
$c_j$ are constants. If $f:R\to \Omega$ is a conformal harmonic map from a 
connected parabolic conformal surface $R$ then $\ell_j\circ f:R\to (-\infty,c_j)$ 
is a bounded from above harmonic function for every $j=1,\ldots,n-1$,
hence constant. Thus, $f(R)$ lies in a real line and hence $f$ is constant.
\end{proof}

%
%   FLEXIBILITY
%
A property of a domain directly opposite to weak or strong hyperbolicity is {\em flexibility} 
for minimal surfaces, introduced in \cite[Definition 1.1]{DrinovecForstneric2023JMAA}. 
A domain $\Omega\subset \R^n$ for $n\ge 3$ is said to be flexible if, given an open 
conformal surface $M$, a compact set $K\subset M$ whose complement has no 
relatively compact connected components, and a conformal harmonic 
immersion $f: U\to \Omega$ from an open neighbourhood of $K$, we can approximate
$f$ uniformly on $K$ by conformal harmonic immersions $M\to\Omega$. 
A flexible domain admits many conformal harmonic images of $\C$,
so it is not weakly hyperbolic. % Conversely, a weakly hyperbolic domain is not flexible.
In particular, the domains in Theorem \ref{th:main} for $m=2$ are not flexible.
A halfspace is neither (weakly) hyperbolic nor flexible.

It is not surprising that the situation regarding Corollary \ref{cor:main3} is quite 
different in $\R^4$. As shown in \cite[Example 1.9]{DrinovecForstneric2023JMAA}, 
a domain in $\R^4$ with coordinates $(x_1,x_2,x_3,x_4)$, given by
\[
	\Omega=\{x_4 > -a|x_2| + b|x_3| \ \ \ \text{for some $a>0$ and $b\in\R$}\},
\]
is flexible. Taking $b>0$, the complementary domain 
$\Omega'=\R^4\setminus \overline{\Omega}$ 
is of the same type with the reversed roles of $x_2$ and $x_3$, and 
with $x_4$ replaced by $-x_4$, so it is also flexible. By 
Alarc\'on and L\'opez \cite{AlarconLopez2012JDG} each of these domains contains 
a properly immersed conformal minimal surface parameterised by an arbitrary 
open Riemann surface. (In fact, their result holds for any concave wedge in $\R^3$ 
obtained by intersecting $\Omega$ with the hyperplane $x_3=0$.) 
This gives many pairs of disjoint properly immersed minimal surfaces 
in $\R^4$ of any given conformal type. There also exist pairs of disjoint catenoids 
in $\R^4$ whose ends are asymptotic to a pair of orthogonal $2$-planes in $\R^4$, 
so their closures in $\RP^4$ are disjoint as well. On the other hand, a pair of complex 
algebraic curves in $\C^2$ intersect, or their closures in $\CP^2$ intersect at infinity. 
(Note that every complex curve is also a minimal surface.)

%
%
%   ACKNOWLEDGMENTS
%
%
\subsection*{Acknowledgements}
I thank Antonio Alarc\'on and Joaqu\'{\i}n P\'erez for information 
on properly embedded minimal surfaces with bounded curvature in $\R^3$, 
and for general advice. I also thank G.\ Pacelli Bessa and Leandro F.\ Pessoa for communication regarding their paper \cite{BessaJorgePessoa2021} and 
the paper by Gama et al.\ \cite{GamaLiraMariMedeiros2022}.
Finally, my sincere thanks go to an anonymous referee for the comments and 
suggestions which led to an improved presentation.

%{\bibliographystyle{abbrv} \bibliography{references}}
%\begin{comment}

%\end{comment}

\end{document}